\documentclass[11pt,leqno,a4paper]{amsart}
\usepackage{amsmath}
\usepackage{amsthm}
\usepackage{amssymb}
\usepackage{amscd}
\usepackage{mathabx}
\usepackage{accents}

\usepackage{tikz}
\usetikzlibrary{matrix,arrows,decorations.pathmorphing}

\usepackage{enumerate}
\usepackage{verbatim}

\usepackage{hyperref}
\usepackage{mathabx}

\usepackage[english]{babel}
\usepackage[utf8]{inputenc}
\usepackage[a4paper,textheight=22cm,textwidth=16cm,centering]{geometry}
\usepackage{mathrsfs}
\usepackage[all]{xy}

\newtheorem{theorem}{Theorem}

\newtheorem{lemma}[theorem]{Lemma}

\newtheorem{prop}[theorem]{Proposition}

\theoremstyle{remark}

\theoremstyle{definition}

\newtheorem{remark}[theorem]{Remark}

\numberwithin{equation}{theorem}

\def\beq{\begin{equation}}
\def\eeq{\end{equation}}

\def\beqn{\begin{equation*}}
\def\eeqn{\end{equation*}}

\def\ben{\begin{enumerate}}
	\def\een{\end{enumerate}}

\def\crash#1{}
\def\A{{\mathbb A}}
\def\B{{\mathbb B}}

\def\N{{\mathbb N}}

\def\R{{\mathbb R}}

\def\Z{{\mathbb Z}}

\def\l{\left}
\def\r{\right}

\def\ie{\emph{i.e.}~}

\def\cO{{\mathcal O}}

\def\sF{{\mathscr F}}

\def\limpro{\mathop{\lim\limits_{\displaystyle\leftarrow}}}

\def\limind{\mathop{\lim\limits_{\displaystyle\rightarrow}}}

\def\bBorn{\mathbf{Born}}

\def\bLoc{\mathbf{Loc}}

\def\wotimes{\widehat{\otimes}}

\newcommand{\msp}[1]{\mathrm{Sp}\left(#1\right)}
\newcommand{\weak}[1]{\langle #1\rangle^\dagger}
\newcommand{\tate}[1]{\langle #1 \rangle}
\newcommand{\norm}[1]{\left\vert#1\right\vert}

\newcommand{\minus}{\scalebox{0.55}[1.0]{$-$}}

\title{Errata corrige to Theorems A and B for dagger quasi-Stein spaces}

\begin{document}
	

\author{Federico Bambozzi}
\address{ Dipartimento di Matematica ``Tullio Levi-Civita" \\ University of Padova \\
Via Trieste 63 \\ 35121 \\
Padova \\ Italy}
\email{federico.bambozzi@unipd.it}

\author{Christopher Lazda}
       \address{Department of Mathematics\\ Harrison Building \\ Streatham Campus
 \\ University of Exeter \\ North Park Road \\ Exeter  \\ EX4 4QF \\ United Kingdom }
       \email{c.d.lazda@exeter.ac.uk}

	\maketitle

In \cite{Bam} the first named author claimed a proof of the following dagger analogue of Cartan's famous Theorems A and B for complex analytic Stein spaces.

\begin{theorem} \label{thm:main}
Let $X$ be a dagger quasi-Stein space over a non-Archimedean field $k$ that can be embedded as a closed subspace
in a finite direct product of a finite number of open polydisks, affine lines, and dagger closed polydisks, and $\sF$ a coherent sheaf on $X$. Then:
\begin{enumerate} \item $\sF$ is generated by its global sections;
\item ${\rm H}^q(X,\sF)=0$ for all $q>0$.
\end{enumerate}
\end{theorem}

We refer to \cite{Bam} for the terminology used in the statement of the theorem. In fact, the second of these (Theorem B) implies the first (Theorem A), and it was Theorem B for which a claimed proof was given in \cite{Bam}. Unfortunately, Theorem B is false for what is perhaps the main non-trivial example of a dagger quasi-Stein space: 

\begin{theorem} \label{theo: main} Suppose that $k$ is non-Archimedean, and let $X=\B_k^{-}\times_k \B^{+}_k$ denote the product of the (dagger) open unit disc with the (dagger) closed unit disc over $K$. Then ${\rm H}^1(X,\cO_X)\neq 0$. 
\end{theorem}

To prove this, fix an increasing sequence of elements $\eta_n\in \sqrt{\norm{k^\times}}\subset \R_{>0}$ in the divisible closure of the value group of $k$, tending towards $1$ from below, and then set $U_n=\msp{k\weak{\eta_n^{-1}x,y}}$. Thus each $U_n$ is a Weierstrass domain in $U_{n+1}$, and $X=\bigcup_n U_n$. Since each $U_n$ is (dagger) affinoid, it follows from \cite[Proposition 3.1]{GK00} that
\[ \mathbf{R}\Gamma(X,\cO_X) \simeq \mathbf{R}\! \limpro_{n \in \N}\mathbf{R}\Gamma(U_n,\cO_{U_n})\simeq \mathbf{R}\!\limpro_{n\in \N}k\weak{\eta_n^{-1}x,y}, \]
and we will show that $\limpro_{n\in \N}^{(1)}\,k\weak{\eta_n^{-1}x,y}\neq 0$ by using a criterion from \cite{Wen03}. 

To state this criterion, let $\{V_n,\pi_{m,n}\colon V_m\to V_n \}_{n\in \N,m\geq n}$ be an $\N$-indexed inverse system of vector spaces over $k$. For locally convex vector spaces over Archimedean fields, the following appears as a consequence of \cite[Proposition 3.2.6]{Wen03}, and the proof here is copied more or less word-for-word from there.

\begin{prop} 
Assume that each $V_n$ is a countable union of $\cO_k$-submodules $A_{n,N}\subset V_n$. If $\limpro_{n\in \N}^{(1)}\,V_n =0$, then there exists a sequence $N_n\in \N^{\N}$, such that for all $n\in \N$ there exists some $m\geq n$, such that for all $l\geq m$, $\pi_{m,n}(V_m)\subset \pi_{l,n}(V_l)+A_{n,N_n}$.
\end{prop}

\begin{proof} The derived limit $\mathbf{R}\!\limpro_{n} V_n$ is computed by the formula
\[ \mathbf{R}\!\limpro_{n} V_n \simeq \left[ \prod_n V_n \overset{\Delta}{\longrightarrow} \prod_n V_n \right] \]
where $\Delta(v_n)=(v_n-\pi_{n+1,n}(v_{n+1}))$. Thus the hypothesis that $\limpro_{n}^{(1)}\,V_n=0$ is equivalent to saying that $\Delta$ is surjective. Let $W$ denote $\prod_n V_n$ equipped with the product topology from the discrete topology on each $V_n$. Thus
\[ W = \Delta\left (\bigcup_N A_{0,N} \times \prod_{n\geq 1} V_n \right)=\bigcup_N \;\Delta\left(A_{0,N}\times \prod_{n\geq 1} V_n \right).\]
Since $W$ is a product of complete metric spaces, it is a Baire space, and it therefore follows that there exists $N_0$ such that the closure of 
\[ \Delta\left(A_{0,N_0}\times \prod_{n\geq 1} V_n\right) \]
has non-empty interior. Proceeding inductively, there exists a sequence $N_j$ such that for all $n$, the closure of 
\[ B_n:=\Delta\left(\prod_{j=0}^n A_{j,N_j}\times \prod_{j> n} V_j\right) \]
has non-empty interior. In particular, since $\overline{B}_n$ contains a translate of a vector subspace of $W$ of the form $\prod_{j<m} \{0\}\times \prod_{j\geq m} V_j$, and is a sub-$\cO_k$-module of $W$, it follows that $\overline{B}_n$ actually contains this vector subspace itself (and not just a translate). Increasing $m$, we might as well assume that $m\geq n$.

Now, $\overline{B}_n$ is contained in
\[  \bigcap_{l\geq m} \left( B_n + \prod_{j<l} \{0\} \times \prod_{j\geq l} V_j\right), \]
so it follows that for all $n$, there exists $m\geq n$ such that for all $l\geq m$, 
\[\prod_{j<m} \{0\}\times \prod_{j\geq m} V_j \subset \Delta\left(\prod_{j=0}^n A_{j,N_j}\times \prod_{j> n} V_j\right) +   \prod_{j<l} \{0\} \times \prod_{j\geq l} V_j.  \]
The claim is now that for any such $l \geq m\geq n$, the inclusion
\[  \pi_{m,n}(V_m)\subset \pi_{l,n}(V_l)+A_{n,N_n} \]
holds. This is clear if $l=m$, so we may as well assume that $l>m$. In this case, let $v_m\in V_m$, and consider $v=(0,\ldots,0,v_m,0,\ldots)\in W$. Then there exists $w=(w_j) \in \prod_{j=0}^n A_{j,N_j}\times \prod_{j> n} V_j$ such that
\[ v-\Delta(w) \in \prod_{j<l} \{0\} \times \prod_{j\geq l} V_j.\]
It follows that $w_j=\pi_{j+1,j}(w_{j+1})$ for all $j<l$, $j\neq m$, and that $v_m=w_m-\pi_{m+1,m}(w_{m+1})$. Therefore
\begin{align*} \pi_{m,n}(v_m) &=  \pi_{m,n}(w_m)-\pi_{m+1,n}(w_{m+1}) \\
&= \pi_{m-1,n}(w_{m-1})-\pi_{m+2,n}(w_{m+2}) \\
& \;\;\vdots \\
&= w_n - \pi_{l,n}(w_l).
\end{align*}
Clearly $\pi_{l,n}(w_l)\in \pi_{l,n}(V_{l})$, and by assumption $w_n\in A_{n,N_n}$, thus
\[ \pi_{m,n}(v_m) \in \pi_{l,n}(V_l)+A_{n,N_n} \]
as required.
\end{proof}

In particular, if $V_n=\limind_{N\in \N} V_{n,N}$ is a countable direct limit of vector subspaces $V_{n,N}$, with injective transition maps, and $\limpro_{n\in\N}^{(1)}\,V_n = 0$, then the following condition is satisfied:
\vspace{2mm}
\begin{itemize}\item for all $n$ there exist $N\in \N$ and $m\geq n$, such that for all $l\geq m$, $\pi_{m,n}(V_m)\subset \pi_{l,n}(V_l)+V_{n,N}$. 

\end{itemize}
\vspace{2mm}
We will apply this to $V_n= k\weak{\eta_n^{-1}x,y}$, which is, essentially by definition, a countable direct limit of certain natural subspaces. Concretely, 
\[ V_n = \limind_{\eta_n<\eta,1<\lambda}k\tate{\eta^{-1}x,\lambda^{-1}y},  \]
and the index set here has a countable cofinal subset. We deduce that if $\limpro_{n\in\N}^{(1)}\,k\weak{\eta_n^{-1}x,y}$ vanishes, then the following condition is satisfied:
\vspace{2mm}
\begin{itemize}\item for all $n$ there exist $m\geq n$, $\lambda>1$, and $\eta>\eta_n$, such that for all $l\geq m$, 
\[  k\weak{\eta_m^{-1}x,y} \subset k\weak{\eta_l^{-1}x,y} +  k\tate{\eta^{-1}x,\lambda^{-1}y}. \]
\end{itemize}
\vspace{2mm}
It is easy to show that this condition is not met, thus completing the proof of Theorem \ref{theo: main}.

\begin{lemma} \label{counterexample}
Suppose that $n\in \N$, $m\geq n$, $\lambda>1$, and $\eta>\eta_n$. Then
\[  k\weak{\eta_m^{-1}x,y} \not\subset k\weak{\eta_{m+1}^{-1}x,y} +  k\tate{\eta^{-1}x,\lambda^{-1}y}. \]
\end{lemma}

\begin{proof} Choose $d>0$ such that $\eta\lambda^d>\eta_m$, choose $\rho$ such that $\min\{\eta_{m+1},\eta\lambda^d\}>\rho>\eta_m$, and choose a sequence $a_i\in k$ such that $\norm{a_i}\rho^i\rightarrow 0$ but $\norm{a_i}\rho'^i\not\rightarrow 0$ for all $\rho'>\rho$. These choices can clearly be made. If we also choose $\delta>1$ and $\eta'>\eta_m$ such that $\eta'\delta^d=\rho$, then the function $\sum_i a_i x^iy^{di}$ lies in $k\tate{\eta'^{-1}x,\delta^{-1}y}$, and hence in $k\weak{\eta_m^{-1}x,y}$, but it cannot lie in
\[ k\weak{\eta_{m+1}^{-1}x,y} +  k\tate{\eta^{-1}x,\lambda^{-1}y}\subset k\tate{\eta_{m+1}^{-1}x,y}+ k\tate{\eta^{-1}x,\lambda^{-1}y} \]
because $\norm{a_i}\min\{\eta_{m+1},\eta\lambda^d\}^i \not\rightarrow 0$.
\end{proof}

\begin{remark} Considering instead de\thinspace Rham cohomology, it is easy to see via direct calculation that the complexes $\Gamma(U_n,\Omega^\bullet_{U_n/k})$, as well as $\Gamma(X,\Omega^\bullet_{X/k})$, are all quasi-isomorphic to $k$ concentrated in degree $0$. Thus
\[ {\bf R}\Gamma(X,\Omega^\bullet_{X/k}) =  \Gamma(X,\Omega^\bullet_{X/k}),  \]
so we recover a version of Theorem B in de\thinspace Rham cohomology. We do not know whether or not to expect this in general for a vector bundle with integrable connection on a quasi-Stein dagger space, however, the existence of overconvergent isocrystals with infinite dimensional de\thinspace Rham cohomology \cite[\S4.2]{LS07} makes it seem unlikely.  
\end{remark}

\begin{remark}
One might be tempted to try to recover a version of Theorem B by considering $\cO_X$ as a sheaf valued in the category ${\bf Mod}_{k_{\blacksquare}}$ of solid $k$-vector spaces in the sense of \cite{CS19}. In other words, we could try to compute $\mathbf{R}\!\limpro_{n\in \N} k\weak{\eta_n^{-1}x,y}$ in ${\bf Mod}_{k_{\blacksquare}}$ rather than in the category ${\bf Mod}_k$ of abstract $k$-vector spaces. However, since products in ${\bf Mod}_{k_{\blacksquare}}$ are exact, it follows that the derived limit of an inverse system $\{V_n\}_{n\in \N}$ of solid $k$-vector spaces is computed via the same formula
\[ \mathbf{R}\!\limpro_{n\in \N} V_n\simeq \left[ \prod_n V_n\overset{\Delta}{\longrightarrow} \prod_n V_n \right] \]
as that of an inverse system of abstract $k$-vector spaces. Since the functor  
\[ {\bf Mod}_{k_{\blacksquare}} \to {\bf Mod}_k \]
taking the underlying vector space is exact and preserves products, we deduce that it commutes with taking $\limpro_{n}^{(1)}$ of an $\N$-indexed inverse systems of solid $k$-vector spaces. In particular, since ${\rm H}^1(X,\cO_X)\neq 0$ as an abstract $k$-vector space, it follows that ${\rm H}^1(X,\cO_X)\neq 0$ as a solid $k$-vector space. 
\end{remark}

In the rest of this note, we explain the mistake in proof of Theorem \ref{thm:main}, and give some non-trivial (that is, non-Stein) examples of dagger quasi-Stein spaces for which Theorem B still holds. The key point is that Lemma 4.18 of \cite{Bam} does not always hold. Indeed, this lemma claims the following: let $W$ be a nuclear LB-space and $\{ V_n \}_{n \in \N}$ be a projective system of nuclear LB-spaces such that $\limpro_{n \in \N} V_n$ is a nuclear Fr\'echet space and $\limpro_{n \in \N} V_n \simeq \limpro_{n \in \N} V_n'$, where $\limpro_{n \in \N} V_n'$ is an epimorphic system of Banach spaces. Then
\[ \mathbf{R} \!\limpro_{n \in \N} W \wotimes_k V_n \simeq \limpro_{n \in \N} W \wotimes_k V_n \simeq W \wotimes_k  \limpro_{n \in \N} V_n  \]
holds in $D(\bLoc_k)$ and $D(\bBorn_k)$, where $\bLoc_k$ and $\bBorn_k$ are the categories of locally convex spaces and bornological spaces over $k$ respectively. The reader unfamiliar with the terminology of the lemma is referred to \cite{Bam}. The isomorphisms are claimed to be given by the canonical morphisms between the objects. 
The problem in the proof of the lemma is that it crucially depends on the claim that the functor $W \wotimes_k (\minus)$ is exact. But in Lemma 4.18 it is only checked that it preserves monomorphisms and strict monomorphisms, giving for granted that it is right exact. More precisely, on the third line of page 728, it is claimed that "Since $W$ is nuclear, the functor $W \wotimes_k (\minus)$ is exact [in $\bBorn_k$ and in $\bLoc_k$] ..". The exactness claim in the category $\bBorn_k$ is true, but in $\bLoc_k$ it is not true because the completed tensor product may not preserve cokernels. One example when this happens is when the input of $W \wotimes_k (\minus)$ is a (even nuclear) Fr\'echet space, like in the case of the lemma. The exactness in $\bLoc_k$ is crucial for the proof of the lemma, which uses the fact that completed tensor products commute with cofiltered limits in ${\bLoc}_k$. Since this fails in ${\bBorn}_k$, one cannot use the exactness of $W \wotimes_k (\minus)$ in ${\bBorn}_k$ to correct the proof of the lemma. 

Indeed, we can interepret the above calculation that ${\rm H}^1(X,\cO_X)\neq 0$ when $X=\B^-_k\times_k \B^+_k$ as a failure of $W \wotimes_k (\minus)$  to commute with cofiltered limits in ${\bBorn}_k$. If we write 
\[ \cO(\B_k^-) = \limpro_{\rho < 1} k \tate{ \eta^{-1} x }  \]
for the Fr\'echet algebra of analytic functions on the open unit disk, and
\[ \cO(\B_k^+) = \limind_{\lambda > 1} k \tate{ \lambda^{-1} y } \]
for the LB-algebra of overconvergent analytic functions on the closed unit disk, then the canonical map of bornological spaces
\[ \cO(\B_k^+) \wotimes_k \cO(\B_k^-) \to \limpro_{\eta < 1} \cO(\B_k^+) \wotimes_k k \tate{ \eta^{-1} x } \]
is \emph{not} an isomorphism. The algebra on the right-hand side is the algebra of analytic functions on the direct product of the open disk with the dagger closed disk $\B_k^+ \times \B_k^-$, whereas the algebra on the left-hand side can be written as 
\[ \cO(\B_k^+) \wotimes_k \cO(\B_k^-) \simeq \limind_{\lambda > 1} ( k \tate{ \lambda^{-1} y } \wotimes_k \cO(\B_k^-) ). \] 
This corresponds to the direct limit of the algebras of analytic functions on a certain countable family of neighborhoods of $\B_k^+ \times_k \B_k^-$ inside $\A_k^2$. If this family of neighborhoods were to form a cofinal system of neighborhoods of $\B_k^+ \times_k \B_k^-$ in $\A_k^2$ then the two algebras would agree. Lemma \ref{counterexample} is essentially based on the fact that this is not the case.


Besides this serious problem, the proof of Theorem \ref{thm:main} is correct for dagger quasi-Stein spaces for which the statement of Lemma 4.18 holds.

\begin{theorem} \label{thm:main_correct}
Let $X = \bigcup_{n \in \N} U_n$ be a dagger quasi-Stein space, written as a union of dagger affinoid subspaces. Suppose that
\[ \limpro_{n \in \N} \cO_X(U_n) \simeq \mathbf{R}\! \limpro_{n \in \N} \cO_X(U_n) \]
in $D(\bBorn_k)$. Then, every coherent sheaf on $X$ satisfies Theorems A and B.
\end{theorem}
\begin{proof}
The proof of Theorem \ref{thm:main} given in \cite{Bam} is correct under the new hypothesis. We comment more on this. Roughly speaking, the proof of Theorem \ref{thm:main} is structured as follows. The last part of the proof shows how to reduce the case of a generic coherent sheaf on $X$ to that of the structure sheaf $\cO_X$. So, it is only needed to check that $\cO_X$ does not have cohomology in higher degrees. Lemma 4.18 wrongly claimed this for the product of the open disk and the closed dagger disk. However, the hypothesis that the complex
\[ \mathbf{R}\! \limpro_{n \in \N} \cO_X(U_n) \]
has the left heart cohomology only in degree zero and is precisely isomorphic to $\limpro_{n \in \N} \cO_X(U_n)$ is equivalent to say that the \u{C}ech cohomology of $\cO_X$ vanishes in higher degrees (see Section 4.1 of \cite{Bam} for more detail on this).
\end{proof}

We conclude by remarking that there are non-trivial examples of spaces satisfying the hypothesis of Theorem \ref{thm:main_correct}. A simple family of such examples is that of the half-open annuli, \ie spaces of the form
\[ \B_{r, R}^{-, +} = \{ x \in k \,|\,  r < |x| \le R  \}, \ \, r, R \in \R_+, \ \ r < R  \] 
where on the closed part of the disk the overconvergent analytic functions are considered. Checking that the condition 
\[ \limpro_{n \in \N} \cO_X(U_n) \simeq \mathbf{R}\! \limpro_{n \in \N} \cO_X(U_n) \]
holds in this case for an exhaustion by Weierstrass subdomains is easy because the `open' and `dagger' directions are disjoint and hence the limit is trivial in the `dagger direction'. More precisely,
\[ \cO_{\B_{r, R}^{-, +}}(\B_{r, R}^{-, +}) = \l \{\left. \sum_{i \in \Z} a_i T^i\;\; \right\vert\;\; \lim_{i \in \N} |a_{-i}| \rho^{-i} = 0, \forall \rho < r,  \exists \rho > R, \lim_{i \in \N} |a_i| \rho^{i} = 0 \r \}. \]
But this is just a direct sum of a Fr\'echet space and an LB-space:
\[ \cO_{\B_{r, R}^{-, +}}(\B_{r, R}^{-, +}) = F \oplus L \]
where 
\begin{align*} F &= \l \{\left. \sum_{i \in \N} a_{-i - 1} T^{-i - 1}\;\; \right\vert \;\;\lim_{i \in \N} |a_{-i -1}| \rho^{-i-1} = 0, \forall \rho < r  \r \}\\ L &= \l \{\left. \sum_{i \in \N} a_i T^i \;\;\right\vert\;\; \exists \rho > R, \lim_{i \in \N} |a_i| \rho^{i} = 0 \r \}, \end{align*}
and
\[ F \simeq \limpro_{\rho < r} F_\rho = \limpro_{\rho < r} \l \{\left. \sum_{i \in \N} a_{-i - 1} T^{-i - 1} \;\;\right\vert \;\; \lim_{i \in \N} |a_{-i -1}| \rho^{-i-1} = 0 \r \}. \]
But it is an easy application of \cite[Lemma 3.23]{Bam} that
\[ F \cong \mathbf{R}\!\limpro_{\rho < r} F_\rho.  \]
Then,
\[ \cO_{\B_{r, R}^{-, +}}(\B_{r, R}^{-, +}) \simeq \limpro_{\rho < r} (F_\rho \oplus L) \simeq  \limpro_{\rho < r} (F_\rho) \oplus L \simeq  \mathbf{R}\! \limpro_{\rho < r} (F_\rho) \oplus L \simeq  \mathbf{R}\! \limpro_{\rho < r} (F_\rho \oplus L). \]

\end{document}